\documentclass[11pt,leqno,twoside]{amsart}
\usepackage{amssymb,amsmath,amsthm,soul,color}
\usepackage{t1enc}
\usepackage[cp1250]{inputenc}
\usepackage{a4,indentfirst,latexsym}
\usepackage{graphics}
\usepackage{mathrsfs}
\usepackage{cite,enumitem,graphicx}
\usepackage[colorlinks=true,urlcolor=blue,
citecolor=red,linkcolor=blue,linktocpage,pdfpagelabels,
bookmarksnumbered,bookmarksopen]{hyperref}
\usepackage[english]{babel}
\usepackage[left=2.61cm,right=2.61cm,top=2.72cm,bottom=2.72cm]{geometry}
\usepackage[hyperpageref]{backref}
\usepackage[colorinlistoftodos]{todonotes}

\makeatletter
\providecommand\@dotsep{5}
\def\listtodoname{List of Todos}
\def\listoftodos{\@starttoc{tdo}\listtodoname}
\makeatother

\numberwithin{equation}{section}
\newtheorem{Th}{Theorem}[section]
\newtheorem{Prop}[Th]{Proposition}
\newtheorem{Lem}[Th]{Lemma}
\newtheorem{Cor}[Th]{Corollary}

\newtheorem{Rem}[Th]{Remark}

   \newcommand{\vp}{\varphi}
   
   \newcommand{\eps}{\varepsilon}

   \def\R{\mathbb{R}}

   \def\J{\mathcal{J}}

   \def\iO{\int_{\Omega}}
   \def\RN{\mathbb{R}^N}
   \def\n{\nabla}

\title[Ground states for a critical system]{Positive ground states for a system of Schr\"odinger equations with critically growing nonlinearities}

\author[P. d'Avenia]{Pietro d'Avenia}
\author[J. Mederski]{Jaros\l aw Mederski}

\address[P. d'Avenia]{\newline\indent
Dipartimento di Meccanica, Matematica e Management
\newline\indent 
Politecnico di Bari
\newline\indent
Via Orabona 4,  70125  Bari, Italy}
\email{\href{mailto:pietro.davenia@poliba.it}{pietro.davenia@poliba.it}}

\address[J. Mederski]{\newline\indent 
Faculty of Mathematics and Computer Science
\newline\indent 
Nicolaus Copernicus University
\newline\indent
ul. Chopina 12/18, 87-100 Toru\'n}
\email{\href{mailto:jmederski@mat.umk.pl}{jmederski@mat.umk.pl}}

\thanks{The first author has been supported by Gruppo Nazionale per l’Analisi Matematica, la Probabilit\`a e le loro Applicazioni (GNAMPA) of Istituto Nazionale di Alta Matematica (INdAM)}
\date{\today}

\begin{document}
\begin{abstract}
We study the following problem
\[
\begin{cases}
-\Delta u = \lambda u + u^{2^*-2} v
&
\hbox{in } \Omega,\\
-\Delta v=  \mu v^{2^*-1} + u^{2^*-1} 
&
\hbox{in } \Omega,\\
u> 0,v> 0
&
\hbox{in } \Omega,\\
u=v=0
&
\hbox{on } \partial \Omega,
\end{cases}
\]
where $\Omega$ is a bounded domain of $\R^N$, $N\geq 4$, $2^*=2N/(N-2)$, $\lambda\in\R$ and $\mu\geq 0$ and we obtain existence and nonexistence results, depending on the value of the parameters $\lambda$ and $\mu$.
\end{abstract}

\maketitle



\section{Introduction}

In the last years, nonlinear elliptic systems have been intensively studied by many authors and results, also for semiclassical states and in the singularly perturbed settings, have been obtained (see, for instance,  \cite{ACR,AC,BDW,BWW,BL,DWW,IT,LW,MMP,MPS,NTTV,P,PS,Sirakov,WW}
and references therein). This kind of systems appears if we look for solitary waves of suitable time-dependent nonlinear Schr\"odinger systems which arise in many physical problems, especially in nonlinear optics (see e.g. \cite{AA}) and in Hartree-Fock theory (see e.g. \cite{EGBB}).

%

In this paper we deal with the problem
\begin{equation}
\label{eq}
\tag{$\mathcal{P}$}
\begin{cases}
-\Delta u = \lambda u + u^{2^*-2} v
&
\hbox{in } \Omega,\\
-\Delta v=  \mu v^{2^*-1} + u^{2^*-1} 
&
\hbox{in } \Omega,\\
u> 0,v> 0
&
\hbox{in } \Omega,\\
u=v=0
&
\hbox{on } \partial \Omega,
\end{cases}
\end{equation}
where $\Omega$ is a bounded domain of $\R^N$, $N\geq 4$, $2^*=2N/(N-2)$, $\lambda\in\R$ and $\mu\geq 0$. 

If $\mu=0$, problem \eqref{eq} is an $N$-dimensional variant of the critical problem studied in  \cite{AD}, 
where the authors, following the {\em classical} approach in the Schr\"odinger-Poisson or in the Klein-Gordon-Maxwell systems (see \cite{AD} and references therein), use the so-called {\em reduction method}, namely, the second equation has a unique solution for a given $u$ and it is possible to put it in the first equation, reducing the system to a single {\em nonlocal} equation. In \cite{AD}, the energy functional has the Mountain Pass geometry and the classical approach due to Brezis-Nirenberg \cite{BN} can be adopted.

However, if $\mu>0$, the reduction argument can be no longer applied since the map $H^1_0(\Omega)\ni u\mapsto v_{u}\in H^1_0(\Omega)$, where $v_u$ is a solution to the problem 
\begin{equation}\label{eq2}
\begin{cases}
-\Delta v = \mu |v|^{2^*-2}v +  |u|^{2^*-1} &
    \hbox{in } \Omega,\\
    u=0
    &
    \hbox{on } \partial \Omega,
\end{cases}
\end{equation}
is not necessarily well-defined.
Recall, indeed, that if $u\neq 0$ and $\mu>0$, then \eqref{eq2} may have at least two solutions (see \cite{Tarantello}) or 
no solution (see \cite{CR,M,Z}). 

We look for solutions of \eqref{eq} as critical points of the $C^1$-functional $\J:H^1_0(\Omega)\times H^1_0(\Omega)\to\R$ given by
$$\J(u,v)=\frac{1}{2}\int_{\Omega} |\nabla u |^2 
-\frac{\lambda}{2}\int_{\Omega}|u|^2
+\frac{1}{2(2^*-1)}\int_{\Omega}|\nabla v|^2
-\frac{\mu}{2^*(2^*-1)}\int_{\Omega}| v|^{2^*}
-\frac{1}{2^*-1}\int_{\Omega}  |u|^{2^*-1}v.$$

We are interested in {\em nontrivial} solutions of \eqref{eq}, namely solutions $(u,v)\in H^1_0(\Omega)\times H^1_0(\Omega)$ with both $u\not\equiv 0$ and $v\not\equiv 0$.
Actually, in this kind of system, one can consider also the so-called {\em semi-trivial} solutions, i.e.  solutions $(u,0)$ with $u\not\equiv 0$ or $(0,v)$ with $v\not\equiv 0$. We observe that, for our problem \eqref{eq}, in the first case the second equation of \eqref{eq} implies that $u\equiv 0$, while, in the second case, our system \eqref{eq} reduces to the well-known equation
\begin{equation}\label{EqBN}
-\Delta v = \mu v^{2^*-1}, \quad \mu>0
\end{equation}
and the existence of solutions to \eqref{EqBN}
depends on the topology of $\Omega$ (see \cite{BahriCoron,BN}).

In particular, we are interested in positive {\em ground states} of \eqref{eq}, namely solutions that minimize $J$ on the
Nehari manifold
\begin{equation}
\label{eq:defN}
\mathcal{N}:=
\left\{(u,v)\in (H^1_0(\Omega)\times H^1_0(\Omega))\setminus\{(0,0)\}
\;\vert\;
\mathbf{G}(u,v)=(0,0)
\right\},
\end{equation}
where
\[
\mathbf{G}(u,v)=\left(
\|\nabla u\|_2^2-\lambda\|u\|_2^2- \int_{\Omega}  |u|^{2^*-1}v
,
\|\nabla v\|_2^2 - \mu\| v\|_{2^*}^{2^*}
 - \int_{\Omega}  |u|^{2^*-1}v
\right)
\]
and
$\|\cdot\|_p$ stands for the standard norm in $L^p(\Omega)$.

Let
\begin{equation*}
\mathbb{I}_N=
\begin{cases}
[0,\sqrt{6}/9] & \hbox{if } N=4,\\
[0,\mu^*] & \hbox{if } N=5,\\
[0,1] & \hbox{if } N=6,\\
[0,+\infty[ & \hbox{if } N\geq 7,
\end{cases}
\end{equation*}
where $\mu^*>0$ is defined in Theorem \ref{5lim}. Our principal aim is to prove the following result.
\begin{Th}\label{MainTh}
If 
$\mu\in\mathbb{I}_N$ and $\lambda\in(0,\lambda_1(\Omega))$, then
problem \eqref{eq} has a ground state solution.
\end{Th}

Due to the presence of two critical terms in the  functional $\J$, whose sum may change sign, there are some difficulties in estimation of the Mountain Pass level for which Palais-Smale sequences are convergent. Therefore the classical approach by Brezis and Nirenberg in \cite{BN}, seems to be difficult to adopt. 
Moreover employing the Nehari manifold technique for a system of equations like e.g. in \cite{BWW,CZ,CZ2,CZ3,DWW,MMP,Sirakov,WW}
one might expect that for any nontrivial $(u,v)$, there are  unique $s_0,t_0>0$ such that $\J(s_0u,t_0v)=\max_{s,t\geq 0}\J(su,tv)$. However not all functions can be projected on $\mathcal{N}$ due to the sign-changing nonlinearity. 
Thus, in order to obtain Theorem \ref{MainTh}, we proceed as follows. First of all, in Section \ref{SectionLimitProb}, we consider the {\em limit} case ($\Omega=\R^N$ and $\lambda=0$), which, as usual, plays a crucial role in comparison of the ground state levels and we construct ground states for this last problem by means of the Aubin-Talenti instantons \cite{Aub,Tal}. In Subsection \ref{subs21} we provide results concerning the limiting case for $N=4$, in Subsection \ref{subs22} we consider the remaining cases $N\geq 5$. 
Then, in Section \ref{SectionProofMainTh}, we restrict our considerations to a set $\mathcal{A}$ of {\em admissible} pairs  (see \eqref{eq:constr}) such that any function in $\mathcal{A}$ can be projected onto $\mathcal{N}$. Next  we observe that almost all elements of a Palais-Smale sequence of $\J$ are admissible and can be projected on the appropriate Nehari manifold of the limiting problem. This enable us to compare the ground state level with the Mountain Pass level of \eqref{eq} using Lemma \ref{LemmaStep1}. Finally we get a nontrivial weak limit of the Palais-Smale sequence, in 
which $\J$ attains its ground state level. We note that obtaining the positivity of solutions to \eqref{eq} is not straightforward since $\J(u,v)\neq \J(|u|,|v|)$. Moreover, a standard procedure based on replacing $u$ and $v$ by the positive parts $u_+$ and $v_+$ in the nonlinear terms in $\J$ does not work since the obtained functional is not of $C^1$-class. These difficulties are overcome at the end of Section \ref{SectionProofMainTh} by defining a suitable $C^1$-functional $\J_+$ (see \eqref{DefOfJplus}) and replacing Palais-Smale sequences by nonnegative ones. 

Finally Section \ref{SectionNonexisctence} is devoted to the following nonexistence results, which shows, in a certain sense, the optimality of the hypotheses in Theorem \ref{MainTh}. Let
\begin{equation}
\label{muN}
\mu_N:=
\begin{cases}
\displaystyle\frac{2(N-2)}{N+2}\left(\frac{6-N}{N+2}\right)^{\frac{6-N}{2(N-2)}} & \hbox{if } N=4,5,\\
1 & \hbox{if } N=6.
\end{cases}
\end{equation}
We have
\begin{Th}\label{ThNonexistence}
Problem \eqref{eq} has no solution provided that one of the following conditions holds:
\begin{enumerate}
\item \label{131} $\mu>\mu_N$ and $\lambda\leq0$ ($N=4,5,6$);
\item \label{132} $\mu\in\R$ and $\lambda\geq\lambda_1(\Omega)$;
\item \label{133} $\mu\in\R$, $\lambda\leq 0$ and $\Omega$ is smooth and starshaped.
\end{enumerate}
\end{Th}

\noindent In the paper $C$ denotes a generic positive constant which can change from line to line.

\section{The limit problem}\label{SectionLimitProb}

First of all, let us recall some well known facts. Let $S$ be the best constant such that
\begin{equation}
\label{bestsob}
S \left(\int_{\R^N} |u|^{2^*} \right)^{2/{2^*}}
\leq \int_{\R^N}|\nabla u|^2
\quad
\hbox{for all } u \in \mathcal{D}^{1,2}(\R^N)
\end{equation}
and let us consider the Aubin-Talenti instantons 
\[
U_{\eps,y}(x)=[N(N-2)]\left(\frac{\eps}{\eps^2 + | x - y |^2}\right)^{(N-2)/2},
\] 
with $\eps>0,y\in\R^N$ (see \cite{Aub,Tal}). 
It is well known that the functions $U_{\eps,y}\in\mathcal{D}^{1,2}(\R^N)$ are solutions of
\begin{equation}
\label{eq:lim3}
-\Delta u=|u|^{2^*-2}u \hbox{ in }\R^N,
\end{equation}
satisfy
$$\int_{\R^N}|\nabla U_{\eps,y}|^2=\int_{\R^N}|U_{\eps,y}|^{2^*}=S^{N/2}$$
and 
$\{U_{\eps,y}\in \mathcal{D}^{1,2}(\R^N)|,\;\eps>0,y\in\R^N4\}$ consists of all positive solutions of \eqref{eq:lim3}.\\

In order to estimate the energy levels of $\J$, in this section we consider the {\em limit} system
\begin{equation}\label{eq2bis}
\begin{cases}
-\Delta u =  |u|^{2^*-3} uv
&
\hbox{ on }\R^N,\\
-\Delta v = \mu |v|^{2^*-2}v +  |u|^{2^*-1}
&
\hbox{ on }\R^N,\\
u,v\in \mathcal{D}^{1,2}(\R^N)
\end{cases}
\end{equation}
where 
$\mathcal{D}^{1,2}(\R^N)=\left\{ u\in L^{2^*} (\mathbb{R}^N) \;\vert\; |\nabla u| \in L^2 (\mathbb{R}^N) \right\}$, equipped with the norm $(\int_{\mathbb{R}^N} |\nabla \cdot |^2)^{1/2}$.
We look for nontrivial solutions of \eqref{eq2bis} as critical points of the functional
$$\J_0(u,v)=\frac{1}{2}\int_{\R^N}|\nabla u|^2
+\frac{1}{2(2^*-1)}\int_{\R^N}|\nabla v|^2
-\frac{\mu}{2^*(2^*-1)}\int_{\R^N} |v|^{2^*}
-\frac{1}{2^*-1}\int_{\R^N}  |u|^{2^*-1}v$$
defined in $\mathcal{D}^{1,2}(\R^N)\times\mathcal{D}^{1,2}(\R^N)$. In particular, we are interested to ground state solutions of \eqref{eq2bis} of the form $(kU_{\eps,y},lU_{\eps,y})$ with $k,l>0$.
So we consider
$$\mathcal{N}_0:=\left\{(u,v)\in (\mathcal{D}^{1,2}(\R^N)\times \mathcal{D}^{1,2}(\R^N))\setminus\{(0,0)\}\;\vert 
\;\mathbf{G}_0 (u,v)=(0,0)
\right\}$$
where
\[
\mathbf{G}_0 (u,v)
=\left(\int_{\R^N}|\nabla u|^2
- \int_{\R^N}  |u|^{2^*-1}v
,
\int_{\R^N}|\nabla v|^2
-\mu\int_{\R^N} |v|^{2^*}
-\int_{\R^N}  |u|^{2^*-1}v
\right)
\]
and
$$\mathcal{N}'_0:=\left\{(u,v)\in (\mathcal{D}^{1,2}(\R^N)\times \mathcal{D}^{1,2}(\R^N))\setminus\{(0,0)\}
\;\vert\;
H_0 (u,v)=0
\right\}$$
where
\[
H_0 (u,v)
=\int_{\R^N}|\nabla u|^2
+ \frac{1}{2^*-1}\int_{\R^N}|\nabla v|^2
-\frac{2^*}{2^*-1}\int_{\R^N} |u|^{2^*-1}v
-\frac{\mu}{2^*-1}\int_{\R^N} |v|^{2^*}.
\]
Of course 
$\mathcal{N}_0$ and $\mathcal{N}'_0$ are $C^1$-manifolds since, for all $(u,v)\in \mathcal{N}_0$,
\[
\mathbf{G}'_0 (u,v)[u,v]=\left((2-2^*) \int_{\R^N}|\nabla u|^2, (2-2^*) \int_{\R^N}|\nabla v|^2 \right) \neq (0,0)
\]
and, for all $(u,v)\in\mathcal{N}_0'$,
$$
H'_0(u,v)[u,v]
=(2-2^*)\int_{\R^N}|\nabla u|^2 +
\frac{2-2^*}{2^*-1}\int_{\R^N}|\nabla v|^2 \neq 0.$$ 

Let us define 
\[
A:=\inf_{(u,v)\in\mathcal{N}_0}\J_0(u,v)
\qquad
\hbox{and}
\qquad
A':=\inf_{(u,v)\in\mathcal{N}'_0}\J_0(u,v).
\]


Since $\mathcal{N}_0\subset \mathcal{N}'_0$, then $$A'\leq A.$$

In the next subsections we find ground state for \eqref{eq2bis}, we show that $A=A'$ and we evaluate exactly the ground state level.

\subsection{The limit problem for $N=4$}\label{subs21}
In this subsection focus on the case $N=4$.
\begin{Lem}
\label{le:n0nem}
If $0\leq\mu\leq 2\sqrt{3}/9$, then $\mathcal{N}_0\neq\emptyset$.
\end{Lem}
\begin{proof}
Let $u\in \mathcal{D}^{1,2}(\R^4)$, $u > 0$ and $\bar m$ be a strictly positive solution of
\begin{equation*} 
m^3-m+\mu=0 
\end{equation*}
which exists since, by $0\leq\mu\leq 2\sqrt{3}/9$, the function \begin{equation}
\label{f}
f(m)=m^3 - m+ \mu
\end{equation}
satisfies $f(0)=\mu \geq 0$ and, in the minimum point $\sqrt{3}/3$, $f(\sqrt{3}/3)=\mu - 2\sqrt{3}/9\leq 0$. Then
\[
\left(\sqrt{\bar{m }\left(\int_{\R^4}|\nabla u|^2\right)  \left(\int_{\R^4} u^4\right)^{-1} } u,
\sqrt{\left(\int_{\R^4}|\nabla u|^2\right)  \left(\bar{m}\int_{\R^4} u^4\right)^{-1}}u\right)
\in \mathcal{N}_0.
\]
%
%
%
%
\end{proof}

To state a condition that allows to get $A'=A$, we need the following technicalities.
\begin{Lem}\label{Lempq} 
The system 
\begin{equation}\label{eq_pq}
\left\{
  \begin{array}{lcl}
    kl=1,\\
    \mu l^3+k^3=l\\
    k,l>0
  \end{array}
\right.
\end{equation}
has a solution if and only if $\mu\leq 2\sqrt{3}/9$. In particular, if
$\mu\leq 0$ or $\mu=  2\sqrt{3}/9$, then the solution is unique and, if $0<\mu< 2\sqrt{3}/9$, then system \eqref{eq_pq} has two different solutions.
\end{Lem}
\begin{proof}
We argue as in the Proof of Lemma \ref{le:n0nem}. Indeed the function $f$ has a unique strictly positive zero if $\mu\leq 0$ or $\mu=  2\sqrt{3}/9$ and two strictly positive zeros if $0<\mu< 2\sqrt{3}/9$. Then, denoted with $\bar{m}$ such zeros, we have that 
$(k,l)= (\sqrt{\bar{m}},1/\sqrt{\bar{m}})$ satisfy system \eqref{eq_pq}.
\end{proof}

\begin{Lem}\label{Lempq_estimate}
Let $k,l>0$ satisfy 
\begin{equation}\label{eq:n10}
k^2 + \frac{1}{3} l^2 \leq \frac{4}{3} k^3 l + \frac{\mu}{3} l^4.
\end{equation}
\begin{enumerate}
\item \label{231} If $\mu=0$ then
\[
\frac{4}{3} \leq k^2+\frac{1}{3}l^2.
\]
\item \label{232} If $\mu\in(0,\sqrt{6}/9)$
then 
\begin{equation}
\label{eq:minkl}
k_2^2+\frac{1}{3}l_2^2
=
\min_{i=1,2}\{k_i^2+\frac{1}{3}l_i^2\}
\leq
k^2+\frac{1}{3}l^2
\end{equation}
and
\begin{equation*}
k_2^2+\frac{1}{3}l_2^2
<
\frac{1}{3\mu},
\end{equation*}
where $(k_i,l_i)$ are the solutions of the system \eqref{eq_pq}, $k_1<k_2$ and $l_2<l_1$.
\item \label{232bis} If $\mu=\sqrt{6}/9$ then 
\[
k_2^2+\frac{1}{3}l_2^2
= \frac{1}{3\mu}
\leq
k^2+\frac{1}{3}l^2.
\]
\item \label{233} If $\mu\in(\sqrt{6}/9,2\sqrt{3}/9]$
then 
\[
\frac{1}{3\mu}< k^2+\frac{1}{3}l^2. 
\]
\end{enumerate}
\end{Lem}

\begin{proof}
Let us fix $k,l>0$ satisfying \eqref{eq:n10} and 
\[
\bar{k}=k\sqrt{\frac{3k^2+l^2}{l(4k^3+\mu l^3)}}
\quad \hbox{ and } \quad
\bar{l}=\sqrt{\frac{l(3k^2+l^2)}{4k^3+\mu l^3}}.
\] 
We have that
\begin{equation}
\label{eq:kbark}
0<\bar{k}\leq k, \quad 0<\bar{l}\leq l,
\end{equation}
\[
\frac{k}{l}= \frac{\bar{k}}{\bar{l}},
\]
\begin{equation}
\label{eq:curve}
\bar{k}^2 + \frac{1}{3} \bar{l}^2 = \frac{4}{3} \bar{k}^3 \bar{l} + \frac{\mu}{3} \bar{l}^4
\end{equation}
By \eqref{eq:kbark} we have that
\[
\bar{k}^2 + \frac{1}{3} \bar{l}^2
\leq
k^2+\frac{1}{3} l^2.
\]
So it is sufficient to prove \eqref{eq:minkl} for $(\bar k,\bar l)$.\\
We notice that, since the system
\begin{equation*}
\begin{cases}
\bar k=m\bar l\\
3 \bar k^2 +  \bar l^2 =  4\bar k^3 \bar l  + \mu \bar l^4\\
\bar l,\bar k>0
\end{cases}
\end{equation*}
admits a unique solution 
\begin{equation}
\label{curveparam}
\bar k = m\sqrt{\frac{3m^2+1}{4m^3+\mu}},
\qquad
\bar l= \sqrt{\frac{3m^2+1}{4m^3+\mu}},
\end{equation}
for every $m>0$, the curve given by \eqref{eq:curve} (for $\bar k,\bar l>0$), can be parametrized by $m$ using \eqref{curveparam}.
Thus  we consider
\begin{equation}
\label{eq:psi}
\bar{k}^2 + \frac{1}{3} \bar{l}^2 = \frac{(3m^2+1)^2}{3(4m^3+\mu )} =:\psi(m),
\quad
m>0.
\end{equation}
If $\mu=0$ then the function $\psi$ has a unique global minimum point (on the positive halfline) in $1$ and $\psi(1)=4/3$.\\
If $0<\mu< \frac{\sqrt{6}}{9}$ then
the function $\psi$ admits two critical points $m_1<m_2$ which solve the equation $ m^3 - m +\mu=0$ and the  global minimum $m_2$ (on the positive halfline) satisfies
\[
\psi ( m_2) < \frac{1}{3\mu} = \lim_{m\to 0^+} \psi(m).
\]
Moreover, if we take 
\[
k_i := m_i\sqrt{\frac{3m_i^2+1}{4m_i^3+\mu}}=\sqrt{m_i}
\qquad
\hbox{and}
\qquad
l_i:= \sqrt{\frac{3m_i^2+1}{4m_i^3+\mu}}=\frac{1}{\sqrt{m_i}},
\qquad
i=1,2,
\]
we have that $(k_i,l_i)$ solve system \eqref{eq_pq}, $k_1<k_2$, $l_2<l_1$ and
\[
\psi(m_2)=k_2^2 + \frac{1}{3} l_2^2.
\]  
If $\mu=\frac{\sqrt{6}}{9}$ then, for any $m>0$, 
$$\psi(m_2)= \lim_{m\to 0^+} \psi(m)=\frac{1}{3\mu}.$$
If $\frac{\sqrt{6}}{9}<\mu\leq\frac{2\sqrt{3}}{9}$ then  
$$\psi(m)> \lim_{m\to 0^+} \psi(m)=\frac{1}{3\mu}.$$
\end{proof}
%
Before we prove the main results of this section, we show the following preliminary  properties.
\begin{Prop}
\label{pr:stl}
Let  $\mu\in[0,\sqrt{6}/9)$.
\begin{enumerate}
\item \label{241} If $\mu=0$, $\mathcal{N}_0'$ does not contain semitrivial couples.
\item \label{242} If $\mu\in(0,\sqrt{6}/9)$, $\mathcal{N}_0'$ does not contain semitrivial couples $(u,0)$ and
\begin{equation}
\label{eq:A'}
A'<\inf_{(0,v)\in\mathcal{N}'_0}\J_0(0,v).
\end{equation}
\end{enumerate}
\end{Prop}
\begin{proof}
Statement (\ref{241}) and the first part of (\ref{242}) are obvious. So it remains to prove \eqref{eq:A'}. We notice that $\mathcal{N}'_0$ contains couples $(0,v)$ with $v\in \mathcal{D}^{1,2}(\R^4)$: it is sufficient to take $v=\mu^{-1/2} U_{\eps,y}$.\\
Let $(0,v)\in \mathcal{N}'_0$. We have that
\[
H_0 (0,v)
=\frac{1}{3}\int_{\R^4}|\nabla v|^2
-\frac{\mu}{3}\int_{\R^4} v^4=0
\]
and so
\[
\J_0(0,v)= \frac{1}{12} \int_{\R^4}|\nabla v|^2.
\]
Moreover, for every $s>0$,
\[ 
(t(s)sv,t(s)v)\in\mathcal{N}_0'
\quad
\hbox{with }
t(s):=\left[\frac{(3s^2 + 1)\mu}{4s^3+\mu}\right]^{1/2}
\]
and then
$$A'\leq \J_0(t(s)sv,t(s)v)=\frac{1}{12}\frac{(3s^2+1)^2\mu}
{4s^3+\mu}\int_{\R^4}|\nabla v|^2.$$
Thus, passing to the infimum,
\[
A'\leq \frac{(3s^2+1)^2\mu}{4s^3+\mu}\inf_{(0,v)\in\mathcal{N}'_0}\J_0(0,v)
\]
and we conclude observing that, since $\mu\in(0,\sqrt{6}/9)$,
\[
\frac{(3s^2+1)^2\mu}{4s^3+\mu}\bigg|_{s=2/(9\mu)}<1.
\]
\end{proof}
\begin{Cor}
If $\mu\in[0,\sqrt{6}/9)$ and  $A'$ is attained for some $(u,v)\in\mathcal{N}_0'$, then $u\neq0$ and $v\neq0$.
\end{Cor}
Using the notations introduced before we are ready to prove the following results.
\begin{Th}\label{PropAandA}
If $\mu\in(0,\sqrt{6}/9)$, then, for every $\varepsilon>0$ and $y\in\mathbb{R}^4$, we have that $(k_2U_{\eps,y},l_2U_{\eps,y})$
is a ground state solution of \eqref{eq2bis} and  
$$\J_0(k_2U_{\eps,y},l_2U_{\eps,y})=A=A'=\frac{1}{4}\Big(k_2^2+\frac{1}{3}l_2^2\Big)S^2.$$
\end{Th}
\begin{proof}
Let $\mu\in(0,\sqrt{6}/9)$. Since $(k_2,l_2)$ satisfies the system \eqref{eq_pq}, then it can be easily shown that for every $y\in\mathbb{R}^4$ we have that $\J'_0(k_2U_{\eps,y},l_2U_{\eps,y})=0$ and so
$(k_2U_{\eps,y},l_2U_{\eps,y})\in \mathcal{N}_0$.
Hence 
\begin{equation*}
A'\leq A\leq \J_0(k_2U_{\eps,y},l_2U_{\eps,y})=\frac{1}{4}\Big(k_2^2+\frac{1}{3}l_2^2\Big)S^2.
\end{equation*}
Let $\{(u_n,v_n)\}\subset\mathcal{N}_0'$ be a minimizing sequence, i.e. such that $\J_0(u_n,v_n)\to A'$. We notice that we can assume $u_n\neq 0$ and $v_n\neq 0$. Indeed, if $(u_n,v_n)\in\mathcal{N}_0'$, as observed before (see (\ref{242}) of Proposition \ref{pr:stl}), $v_n\neq 0$ and the existence of a subsequence such that $u_n = 0$ contradicts \eqref{eq:A'}.\\
Since 
\begin{align*}
S\left[\left(\int_{\mathbb{R}^4} u_n^4\right)^{1/2}
+\frac{1}{3} \left(\int_{\mathbb{R}^4} v_n^4\right)^{1/2}\right]
\leq &
\int_{\mathbb{R}^4} |\nabla u_n|^2
+\frac{1}{3} \int_{\mathbb{R}^4} |\nabla v_n|^2
= 
\frac{4}{3} \int_{\mathbb{R}^4} u_n^3 v_n
+ \frac{\mu}{3} \int_{\mathbb{R}^4} v_n^4\\
\leq &
\frac{4}{3}
\left(\int_{\mathbb{R}^4} u_n^4\right)^{3/4}
\left(\int_{\mathbb{R}^4} v_n^4\right)^{1/4}
+ \frac{\mu}{3} \int_{\mathbb{R}^4} v_n^4,
\end{align*}
then
\begin{multline*}
\left[\frac{1}{\sqrt{S}}\left(\int_{\mathbb{R}^4} u_n^4\right)^{1/4}\right]^2 
+ \frac{1}{3} \left[\frac{1}{\sqrt{S}}\left(\int_{\mathbb{R}^4} v_n^4\right)^{1/4}\right]^2\\
\leq 
\frac{4}{3}\left[\frac{1}{\sqrt{S}}\left(\int_{\mathbb{R}^4} u_n^4\right)^{1/4}\right]^3
\left[\frac{1}{\sqrt{S}}\left(\int_{\mathbb{R}^4} v_n^4\right)^{1/4}\right]
+\frac{\mu}{3}  \left[\frac{1}{\sqrt{S}}\left(\int_{\mathbb{R}^4} v_n^4\right)^{1/4}\right]^4.
\end{multline*}
Thus, by (\ref{232}) of Lemma \ref{Lempq_estimate} we get
\[
k_2^2 + \frac{1}{3} l_2^2
\leq
\frac{1}{S} \left[\left(\int_{\mathbb{R}^4} u_n^4\right)^{1/2}
+\frac{1}{3} \left(\int_{\mathbb{R}^4} v_n^4\right)^{1/2}\right].
\]
Therefore 
\begin{align*}
A'+o_n(1)=\J_0(u_n,v_n)= &
\frac{1}{4} \left( \int_{\mathbb{R}^4} |\nabla u_n|^2
+\frac{1}{3} \int_{\mathbb{R}^4} |\nabla v_n|^2 \right)
\geq 
\frac{S}{4} \left[\left(\int_{\mathbb{R}^4} u_n^4\right)^{1/2}
+\frac{1}{3} \left(\int_{\mathbb{R}^4} v_n^4\right)^{1/2}\right]\\
\geq & \frac{1}{4} \Big(k_2^2+\frac{1}{3}l_2^2\Big)S^2
\end{align*}
and thus
$$A'=\frac{1}{4}\Big(k_2^2+\frac{1}{3}l_2^2\Big)S^2.$$
\end{proof}

Proceeding as in the proof of Theorem \ref{PropAandA} and applying (\ref{241}) of Proposition \ref{pr:stl} and (\ref{231}) of Lemma \ref{Lempq_estimate} we can prove
\begin{Th}\label{cor0}
If $\mu=0$, then, for every $\varepsilon>0$ and $y\in\mathbb{R}^4$, we have that $(U_{\eps,y},U_{\eps,y})$ is a ground state solution of \eqref{eq2bis} and
\[
\J_0(U_{\eps,y},U_{\eps,y})
=A=A'
=\frac{S^2}{3}.
\]
\end{Th}


Moreover we have
\begin{Th}\label{ThSecodInterval}
If $\mu\in[\sqrt{6}/9,2\sqrt{3}/9]$, then for every $\varepsilon>0$ and $y\in\mathbb{R}^4$, 
$\left(0,\frac{1}{\sqrt{\mu}} U_{\varepsilon,y}\right)$ is a ground state solution of  \eqref{eq2bis} and 
\[
\J_0\left(0,\frac{1}{\sqrt{\mu}} U_{\varepsilon,y}\right)
=A'=A =\frac{1}{12\mu}S^2.
\]
Moreover if $\mu\in(\sqrt{6}/9,2\sqrt{3}/9]$, then any minimizer $(u,v)$ of $\J_0$ on $\mathcal{N}_0'$ is semitrivial, i.e. $u=0$.
\end{Th}
\begin{proof}
It is simple to verify that $\left(0,\frac{1}{\sqrt{\mu}} U_{\varepsilon,y}\right)$ solves \eqref{eq2bis} and
\[
\J_0\left(0,\frac{1}{\sqrt{\mu}} U_{\varepsilon,y}\right)
=\frac{1}{12\mu}S^2.
\] 
Thus
\[
A'\leq A \leq \frac{1}{12\mu}S^2.
\]
Now we consider a minimizing sequence Let $\{(u_n,v_n)\}\subset\mathcal{N}_0'$ (such that $\J_0(u_n,v_n)\to A'$) and we distinguish two cases: if $u_n\neq 0$ we can proceed as in the proof of Theorem  \ref{PropAandA} getting,
by (\ref{232bis}) and (\ref{233}) of Lemma \ref{Lempq_estimate},
\[
\frac{1}{3\mu}
\leq
\frac{1}{S} \left[\left(\int_{\mathbb{R}^4} u_n^4\right)^{1/2}
+\frac{1}{3} \left(\int_{\mathbb{R}^4} v_n^4\right)^{1/2}\right].
\]
Thus
\begin{equation}
\label{bound}
\J_0(u_n,v_n) \geq \frac{1}{12\mu}S^2.
\end{equation}
If $u_n=0$, since $(0,v_n)\in \mathcal{N}_0'$ and $v_n$ satisfies \eqref{bestsob}, we obtain
\[
\|v_n\|_4^2 \geq \frac{1}{\mu}S
\]
and so the estimate \eqref{bound} holds too.
Thus the first part of the statement is proved.\\ 
Finally, suppose by contradiction that there exists a minimizer $(u,v)\in\mathcal{N}_0'$ of $\J_0$ on $\mathcal{N}_0'$ with $u\neq 0$. Then $v\neq 0$ and similarly as above, 
by (\ref{233}) of Lemma \ref{Lempq_estimate},
\[
\frac{1}{3\mu}
<
\frac{1}{S} \left[\left(\int_{\mathbb{R}^4} u^4\right)^{1/2}
+\frac{1}{3} \left(\int_{\mathbb{R}^4} v^4\right)^{1/2}\right].
\]
Thus
\begin{equation*}
A'
=\J_0(u,v) 
\geq
\frac{S}{4} \left[\left(\int_{\mathbb{R}^4} u^4\right)^{1/2}
+\frac{1}{3} \left(\int_{\mathbb{R}^4} v^4\right)^{1/2}\right]
> \frac{1}{12\mu}S^2=A'
\end{equation*}
and we get a contradiction.
\end{proof}

Finally, by (\ref{232bis}) of Lemma \ref{Lempq_estimate} and Theorem \ref{ThSecodInterval} we have
\begin{Th}
\label{corupp}
If $\mu=\sqrt{6}/9$, then, for every $\varepsilon>0$ and $y\in\mathbb{R}^4$, we have that $(k_2U_{\eps,y},l_2U_{\eps,y})$
is a nontrivial ground state of \eqref{eq2bis}.  
\end{Th}


\subsection{The limit problem for $N\geq 5$} \label{subs22}

In this subsection we study the limit problem for a general $N\geq 5$. We notice that in the previous subsection the key points consist of the existence of a zero of the function $f$ in \eqref{f} (to prove that $\mathcal{N}_0$ is nonempty), the solutions of the system \eqref{eq_pq}, the condition \eqref{eq:n10}, and  the global minimum of the function $\psi$ in \eqref{eq:psi}. For a general $N$, the mentioned issues take the following form
\begin{equation}
\label{fN}
f_N(m)=m^{2^*-1}-m^{2^*-3}+\mu,
\quad
m>0,
\end{equation}
\begin{equation}\label{eq_pqN}
\left\{
  \begin{array}{lcl}
    k^{2^*-3}l=1,\\
    \mu l^{2^*-1}+k^{2^*-1}=l\\
    k,l>0,
  \end{array}
\right.
\end{equation}
\begin{equation}\label{eq:n10N}
k^2 + \frac{1}{2^*-1} l^2 \leq \frac{2^*}{2^*-1} k^{2^*-1} l +
\frac{\mu}{2^*-1} l^{2^*},
\end{equation}
and 
\begin{equation}
\label{eq:psiN}
\psi_N(m)= \frac{((2^*-1)m^2+1)^{\frac{2^*}{2^*-2}}}{(2^*-1)(2^*m^{2^*-1}+\mu )^{\frac{2}{2^*-2}}}
\quad
m>0.
\end{equation}
If $N=5$, the function $f_5$ in \eqref{fN} has the same geometry of the function $f$ in \eqref{f}. Thus we can repeat the similar arguments used in the Subsection \ref{subs21} and we have
\begin{Th}\label{5lim}
There exists $\mu^*\in(0,6/(7\sqrt[6]{7}))$ such that:
\begin{itemize}
\item if $\mu=0$, then, for every $\varepsilon>0$ and $y\in\mathbb{R}^5$, we have that $(U_{\eps,y},U_{\eps,y})$ is a ground state solution of \eqref{eq2bis} and
\[
\J_0(U_{\eps,y},U_{\eps,y})
=A=A'
=\frac{2}{7}S^{5/2};
\]
\item if $\mu\in(0,\mu^*)$, then, for every $\varepsilon>0$ and $y\in\mathbb{R}^5$, we have that $(k_2U_{\eps,y},l_2U_{\eps,y})$
is a ground state solution of \eqref{eq2bis} and  
$$\J_0(k_2U_{\eps,y},l_2U_{\eps,y})=A=A'=\frac{1}{5}\Big(k_2^2+\frac{3}{7}l_2^2\Big)S^{5/2}$$
where $(k_i,l_i)$ are the solutions of the system \eqref{eq_pqN}, $k_1<k_2$ and $l_2<l_1$;
\item if $\mu=\mu^*$, then, for every $\varepsilon>0$ and $y\in\mathbb{R}^5$, we have that $(k_2U_{\eps,y},l_2U_{\eps,y})$
is a nontrivial ground state of \eqref{eq2bis};
\item  if $\mu\in[\mu^*,6/(7\sqrt[6]{7})]$, then for every $\varepsilon>0$ and $y\in\mathbb{R}^5$, 
$\left(0,\frac{1}{\mu^{3/4}} U_{\varepsilon,y}\right)$ is a ground state solution of  \eqref{eq2bis} and 
\[
\J_0\left(0,\frac{1}{\mu^{3/4}} U_{\varepsilon,y}\right)
=A'=A =\frac{3}{35\mu^{3/2}}S^{5/2}.
\]
Moreover if $\mu\in(\sqrt{6}/9,2\sqrt{3}/9]$, then any minimizer $(u,v)$ of $\J_0$ on $\mathcal{N}_0'$ is semitrivial, i.e. $u=0$.
\end{itemize}
\end{Th}

We notice that, the upper bound on $\mu$ to obtain that $\mathcal{N}_0\neq\emptyset$ for $N=4,5$ is given by $\mu_N$ in \eqref{muN}. For $N\geq 6$ the geometry of function $f_N$ is different and allows us to prove the following results.

\begin{Th}\label{6lim}
If $N=6$ and $\mu\in[0,1]$ then, for every $\varepsilon>0$ and $y\in\mathbb{R}^6$, we have that $(\sqrt{1-\mu}U_{\eps,y},U_{\eps,y})$ is a ground state solution of \eqref{eq2bis} and  
$$\J_0(\sqrt{1-\mu}U_{\eps,y},U_{\eps,y})=A=A'=\frac{1}{6}\left(\frac{3}{2}-\mu\right)S^3.$$
\end{Th}
\begin{proof}
In this case it is easy to prove that, if $u\in \mathcal{D}^{1,2}(\mathbb{R}^6)$, $u>0$, then $(\sqrt{1-\mu}u,u)\in\mathcal{N}_0$ and the couple $(\sqrt{1-\mu},1)$ is the unique solution of the system \eqref{eq_pqN}. Thus, arguing as in Lemma \ref{Lempq_estimate} we can prove that \eqref{eq:n10N} implies
\[
k^2+\frac{1}{2} l^2 \geq \frac{3}{2} - \mu.
\]
Hence, similarly as in Theorem \ref{PropAandA} we conclude.
\end{proof}

\begin{Th}\label{7lim}
If $N\geq 7$ and $\mu\geq 0$ then, for every $\varepsilon>0$ and $y\in\mathbb{R}^N$, we have that
$(\bar m^{\frac{1}{2^*-2}}U_{\eps,y}, \bar m^{\frac{{3-2^*}}{2^*-2}}U_{\eps,y})$ is a ground state solution of \eqref{eq2bis} and  
$$\J_0(\bar m^{\frac{1}{2^*-2}}U_{\eps,y}, \bar m^{\frac{{3-2^*}}{2^*-2}}U_{\eps,y})=A=A'=\frac{1}{N}\left(\tilde{k}^2 + \frac{1}{2^*-1} \tilde{l}^2\right)S^{N/2}$$
where $(\tilde{k},\tilde{l})$ is the unique solution of system \eqref{eq_pqN}.
\end{Th}
\begin{proof}
Take any $u\in \mathcal{D}^{1,2}(\R^N)$, $u> 0$. In this case the function $f_N$ in \eqref{fN} is stricly increasing and satisfies \[
\lim_{m\to 0^+} f(m)=-\infty
\qquad
\hbox{and}
\qquad
\lim_{m\to +\infty} f(m)=+\infty.
\]
Thus it admits a unique nontrivial zero $\bar m$ and then
\[
\left(
\left[\bar m \left(\int_{\R^N}|\nabla u|^2\right)  \left(\int_{\R^N} |u|^{2^*}\right)^{-1}\right]^{\frac{1}{2^*-2}} u,
\left[\bar m^{3-2^*} \left(\int_{\R^N}|\nabla u|^2\right)  \left(\int_{\R^N} |u|^{2^*}\right)^{-1}\right]^{\frac{1}{2^*-2}} u
\right)
\in \mathcal{N}_0
\]
and system \eqref{eq_pqN} has a unique solution 
\[
(\tilde{k},\tilde{l})=(\bar m^{\frac{1}{2^*-2}},
\bar m^{\frac{{3-2^*}}{2^*-2}}).
\]
As before we can prove that \eqref{eq:n10N} implies
\[
k^2 + \frac{1}{2^*-1} l^2 
\geq \tilde{k}^2 + \frac{1}{2^*-1} \tilde{l}^2
= \frac{\bar m^{\frac{{2(3-2^*)}}{2^*-2}}}{2^*-1}((2^*-1)\bar m^2 + 1)
\]
considering the function $\psi_N$ in \eqref{eq:psiN} which has a global minimum point at the zero of the function $f_N$. Hence, arguing as in Theorem \ref{PropAandA} we conclude.
\end{proof}

\section{Positive ground states for \eqref{eq}}\label{SectionProofMainTh}

In this section we investigate the existence of ground states for our problem \eqref{eq} and we prove our main result.
First of all we notice that $\mathcal{N}$, defined in \eqref{eq:defN}, is a $C^1$-manifold since
\[
\mathbf{G}'(u,v)[u,v]=\left((2-2^*) (\| \nabla u \|_2^2 -  \lambda \| u \|_2^2), (2-2^*) \| \nabla v \|_2^2 \right) \neq (0,0)
\]
for all $(u,v)\in \mathcal{N}$.
\begin{Lem}
\label{le:nnpty}
If $\lambda\in(0,\lambda_1(\Omega))$ 
and $\mu \in \mathbb{I}_N$, then $\mathcal{N}\neq\emptyset$.
\end{Lem}
\begin{proof}
We proceed as before. Let us take  $u\in H^1_0(\Omega)$, $u > 0$ and $\bar{m}$ be a strictly positive solution of
\[
m^{2^*-1} - \sigma m^{2^*-3} + \mu = 0,
\qquad
\sigma:=\frac{\| \nabla u \|_2^2}{\| \nabla u \|_2^2  - \lambda \| u \|_2^2}
\]
whose existence  can be obtained arguing as in Lemma \ref{le:n0nem} and using that $\sigma >1$ for $N=4,5$, as in Theorem \ref{6lim} for $N=6$ or as in Theorem \ref{7lim} for $N\geq 7$. 
Then
\[
\left((\bar{m}\bar{\sigma})^\frac{1}{2^*-2}u,(\bar{m}^{3-2^*}\bar{\sigma})^\frac{1}{2^*-2}u\right)\in\mathcal{N},
\qquad
\bar{\sigma}:=\frac{\| \nabla u \|_2^2  - \lambda \| u \|_2^2}{\| u \|_{2^*}^{2^*}}.
\]
\end{proof}
We note that, if $N=4,5$, arguing in the same way, we can prove that $\mathcal{N}\neq\emptyset$ for $\mu\in[0,\mu_N]$, where $\mu_N$ is given by
\eqref{muN}.

Now, let 
\[
\mathcal{B}:=\inf_{w\in\Gamma}\max_{t\in[0,1]}\J(w(t))
\]
where $\Gamma:=\{w\in C([0,1],H^1_0(\Omega)\times H^1_0(\Omega))|\; w(0)=(0,0),\; \J(w(1))<0\}$. We have
\begin{Lem}\label{LemmaStep1}
If $\lambda>0$ and $\mu\in\mathbb{I}_N$, then $\mathcal{B}<A$. 
\end{Lem}
\begin{proof}
Without loss of generality we can assume that $0\in\Omega$. Then there exists $R>0$ such that $\bar{B}_R (0)\subset\Omega$. Let $\chi \in C_0^1(\Omega)$ be a nonnegative function such that $\chi\equiv 1$ on $\bar{B}_R (0)$. For every $\eps>0$ let us define $U_{\eps}=\chi U_{\eps,0}$. By \cite{BN}, see also \cite{Willem}, we have that 
\[
\| \nabla U_{\eps} \|_2^2= S^{N/2} + O(\eps^{N-2}),\qquad
\|  U_{\eps} \|_{2^*}^{2^*}= S^{N/2} + O(\eps^N)
\]
and
\[
\| U_{\eps} \|_2^2 \geq C \vp_N(\eps) + O(\eps^{N-2})\hbox{ for }N\geq 5, 
\]
for some $C>0$, where 
$$\vp_N(\eps)=
\begin{cases}
\eps^2 |\log \eps | & \hbox{if } N=4,\\
\eps^2 & \hbox{if } N\geq 5.
\end{cases}$$
Let $(k,l)\in\mathbb{R}^2$, $k,l>0$ such that $(kU_{\eps,y},lU_{\eps,y})$ is a ground state of the limit problem \eqref{eq2bis}.
and consider $(u_\eps,v_\eps) = (k U_{\eps}, l U_{\eps})$. We have that
\begin{gather*}
\| \nabla u_\eps \|_2^2=k^2 S^{N/2} + O(\eps^{N-2}), \qquad \| \nabla v_\eps \|_2^2=l^2 S^{N/2} + O(\eps^{N-2}),\\
\|  v_{\eps} \|_{2^*}^{2^*}= l^{2^*} S^{N/2} + O(\eps^N),\qquad
\int_\Omega u_\eps^{2^*-1} v_\eps = k^{2^*-1} l S^{N/2} + O(\eps^N)
\end{gather*}
and
\[
\|  u_{\eps} \|_2^2 \geq C \vp(\eps)+O(\eps^{N-2}).
\]
Then, since $(k,l)$ satisfies
\[
k^2 + \frac{1}{2^*-1} l^2 = \frac{2^*}{2^*-1} k^{2^*-1} l +
\frac{\mu}{2^*-1} l^{2^*},
\]
we get
\begin{align*}
\J ( t u_{\eps}, t v_{\eps}) = & \frac{1}{2} t^2 \| \nabla u_\eps \|_2^2 -\frac{\lambda}{2} t^2 \| u_{\eps} \|_2^2 + \frac{1}{2(2^*-1)} t^2 \| \nabla v_\eps \|_2^2 - \frac{\mu}{2^*(2^*-1)} t^{2^*} \|  v_{\eps} \|_{2^*}^{2^*}\\
& - \frac{1}{2^*-1} t^{2^*} \int_\Omega u_\eps^{2^*-1} v_\eps\\
\leq & \frac{1}{2} t^2 \left( \left(k^2+\frac{1}{2^*-1}l^2\right)S^{N/2}
 - \lambda C \vp(\eps) + O(\eps^{N-2})\right)\\
&- \frac{1}{2^*}t^{2^*} \left(\left(k^2+\frac{1}{2^*-1}l^2\right)S^{N/2}
 + O(\eps^N)\right)\\
 =& \frac{1}{2} t^2 ( N A  - \lambda C \vp(\eps) + O(\eps^2))
- \frac{1}{2^*}t^{2^*} (N A  + O(\eps^N)).
\end{align*}
Let us denote
\[
A_\eps = N A  - \lambda C \vp(\eps) + O(\eps^{N-2}),\qquad
B_\eps = N A + O(\eps^N)
\]
and
consider
\[
f(t):= \frac{A_\eps}{2}t^2 - \frac{B_\eps}{2^*}t^{2^*}.
\]
We have that 
$$\max_{t>0}f(t)=\frac{1}{N}\Big(\frac{A_\eps}{B_\eps^{(N-2)/N}}\Big)^{N/2}<A$$
for $\eps>0$ sufficiently small.
Thus
\[
\mathcal{B} \leq \max_{t>0} \J (t u_\eps, t v_\eps) < A.
\]
\end{proof}

Let us consider
\[
\mathcal{N}'=\left\{ (u,v) \in  (H^1_0(\Omega)\times H^1_0(\Omega))\setminus \{(0,0)\} \; \vert \; H(u,v)=0\right\}\]
where
\[
H(u,v)= \| \nabla u \|_2^2 - \lambda  \| u \|_2^2 + \frac{1}{2^*-1}  \| \nabla v \|_2^2 - \frac{\mu}{2^*-1} \| v \|_{2^*}^{2^*} - \frac{2^*}{2^*-1}\int_{\Omega} |u|^{2^*-1}v
\]
and 
\begin{equation}
\label{eq:constr}
\mathcal{A}:=\left\{(u,v)\in (H^1_0(\Omega)\times H^1_0(\Omega))|\;
\mu \| v \|_{2^*}^{2^*} + 2^* \int_{\Omega} |u|^{2^*-1}v >0\right\}
\end{equation}
the set of {\em admissible} pairs.
Note that, if $\lambda\in(0,\lambda_1(\Omega))$, we have that $\mathcal{N}'$ is a $C^1$-manifold being, for all $(u,v)\in\mathcal{N}'$,
\[
H'(u,v)[u,v]=(2-2^*)\left(\|\nabla u\|_2^2 - \lambda \| u \|_2^2 + \frac{1}{2^*-1} \|\nabla v\|_2^2\right) \neq 0.
\]
Moreover
$$\mathcal{N}\subset\mathcal{N'}\subset \mathcal{A}$$
and, in view of the H\"older inequality and the Sobolev embeddings,
\begin{equation}\label{estimatSobG}
H(u,v) \geq \|(u,v)\|^2-C\|(u,v)\|^{2^*}
\end{equation}
for some constant $C>0$, where 
\[
\|(u,v)\|^2:= \| \nabla u \|_2^2 - \lambda  \| u \|_2^2 + \frac{1}{2^*-1}  \| \nabla v \|_2^2.
\]
We have
\begin{Prop}\label{PropEqInfB}
If $\lambda\in(0,\lambda_1(\Omega))$
and $\mu\in\mathbb{I}_N$, then 
$$\inf_{(u,v)\in\mathcal{N}'} \J(u,v)=
\inf_{(u,v)\in \mathcal{A}}\max_{t\geq 0} \J(t u,t v)=
\mathcal{B}>0.
$$
\end{Prop} 
\begin{proof}
Let $(u,v)\in \mathcal{A}$ and 
\begin{equation*} \label{DefOfTbar}
\bar{t}=\left[\left(\| \nabla u \|_2^2 - \lambda  \| u \|_2^2 +\frac{1}{2^*-1}\| \nabla v \|_2^2\right)
\left(\frac{\mu}{2^*-1} \| v \|_{2^*}^{2^*} + \frac{2^*}{2^*-1} \int_{\Omega} |u|^{2^*-1}v\right)^{-1}\right]^{\frac{1}{2^*-2}}.
\end{equation*}
Observe that $(\bar{t}u,\bar{t}v)\in\mathcal{N}'$ and so
$$\mathcal{J} (\bar{t}u,\bar{t}v)\geq \inf_{(u,v)\in\mathcal{N}'} \J(u,v).$$
Moreover $\bar{t}$ is the unique strictly positive real number such that
\[
\mathcal{J} (\bar{t}u,\bar{t}v) =
\max_{t\geq 0} \mathcal{J} (tu,tv).
\]
If $(u,v)\in\mathcal{N}'$, then $\bar{t}=1$ and, since $\mathcal{N'}\subset \mathcal{A}$, we get
$$\inf_{(u,v)\in\mathcal{N}'} \J(u,v)\geq
\inf_{(u,v)\in \mathcal{A}}\max_{t\geq 0} \J(t u,t v).$$
Moreover, since if $(u,v)\in\mathcal{A}$, then there is $t>0$ such that $\J(tu,tv)<0$,
$$\inf_{(u,v)\in \mathcal{A}}\max_{t\geq 0} \J(t u,t v)\geq \mathcal{B}.$$
Let now $w=(w_1,w_2)\in \Gamma$.  We claim that there exists $t_1 > 0$ such that $H(w(t_1))=0$, namely $w(t_1)\in \mathcal{N}'$.
To this end we consider the continuous function
$\varphi:[0,1]\to\R$,
$$\varphi(t)=\mu\|w_2(t)\|_{2^*}^{2^*}+2^*\int_{\Omega}|w_1(t)|^{2^*-1} w_2(t).$$
We have that $\varphi(1)>0$. Let $t_0\in[0,1)$ such that
$\varphi(t_0)=0$ and $\varphi(t)>0$ for $t\in(t_0,1]$.
Observe that
$$H(w(t_0))\geq 0$$
and
$$H(w(1))=2\J(w(1))-\frac{2}{N(2^*-1)}\varphi(1)<0.$$
If $H(w(t_0))>0$, our claim is proved.
If $H(w(t_0))=0$, then $w(t_0)=(0,0)$. Thus, by \eqref{estimatSobG}, $H(w(t'_0))>0$ for some $t_0'\in[t_0,1)$ and we get the claim. Then
$$ \mathcal{B}\geq \inf_{(u,v)\in \mathcal{N}'} 
\J(u,v).$$
Finally, note that if $\J(u_n,v_n)\to 0$ and $(u_n,v_n)\in\mathcal{N}'$
then $\|(u_n,v_n)\|\to 0$ which contradicts the inequality 
\eqref{estimatSobG}.
Thus 
$$\inf_{(u,v)\in\mathcal{N}'} \J(u,v)>0.$$
\end{proof}
We notice that in this last proof we only need that $\mathcal{N}\neq\emptyset$. Then, if $N=4,5$ we can assume $\mu\in[0,\mu_N]$.
\begin{Rem}
In the study of elliptic problems involving the Mountain Pass geometry, usually one expects that a Nehari manifold is homeomorphic to the unit sphere (see e.g. \cite[Lemma 4.1]{Willem}). However, due to the sing-changing nonliearities, this no longer holds in our case, but we have that the map $(u,v)\mapsto (\bar{t}u,\bar{t}v)$, where  $\bar{t}$ is given by \eqref{DefOfTbar}, defines a homeomorphism from $\mathcal{A}\cap S^1$ into $\mathcal{N}'$, with $S^1:=\{(u,v)\in (H^1_0(\Omega)\times H^1_0(\Omega))
|\; \|(u,v)\|=1\}$.
\end{Rem}

Before we prove the main result of this section, we show the following preliminary property.

\begin{Prop}
Let $\lambda\in (0,\lambda_1(\Omega))$ and $\mu\geq 0$. If a ground state $(u,v)$ of \eqref{eq} exists, then $(u,v)$ nontrivial.
\end{Prop}
\begin{proof}
Let $(u,v)\in\mathcal{N}$ be such that 
$$\J(u,v)=\inf_{\mathcal{N}}\J.$$
If $v=0$, then $\langle\J'(u,0),(u,0)\rangle=0$ implies $u=0$.
Now suppose that $u=0$. If $\mu= 0$, then we get easily that $v=0$.
Let $\mu>0$ and then
$v$ is a nontrivial solution to
\[
\left\{
\begin{array}{ll}
    -\Delta v = \mu |v|^{2^*-2}v
    &
    \hbox{in } \Omega,\\
    v=0
    &
    \hbox{on } \partial \Omega.
\end{array}
\right.
\]
Observe that
\begin{align*}
\inf \{\J(0,w)|\; w\in H^1_0(\Omega)\setminus\{0\} ,\|\nabla w\|_2^2=\mu\|w\|_{2^*}^{2^*}\}
& \leq 
\J(0,v) = \inf_{\mathcal{N}}\J\\
& \leq
\inf \{\J(0,w)|\; w\in H^1_0(\Omega)\setminus\{0\} ,\|\nabla w\|_2^2=\mu\|w\|_{2^*}^{2^*}\}
\end{align*}
and
\begin{align*}
\inf \{\J(0,w)|\; w\in H^1_0(\Omega)\setminus\{0\} ,&\|\nabla w\|_2^2=\mu\|w\|_{2^*}^{2^*}\}\\
 = &
\frac{1}{N(2^*-1)} \inf \{\|\nabla w\|_2^2|\; w\in H ^1_0(\Omega)\setminus\{0\}, \|\nabla w\|_2^2=\mu\|w\|_{2^*}^{2^*}\}\\
 = &
\frac{1}{N(2^*-1))\mu^{(N-2)/N}}\inf \{\|\nabla w\|_2^{N}|\; w\in H ^1_0(\Omega),\|w\|_{2^*}=1\}.
\end{align*}
Then 
\[
\bar{v}=\left(\frac{\mu}{\|\nabla v\|_2^2}\right)^{1/{2^*}}v
\]
satisfies $\|\bar{v}\|_{2^*}=1$ and
\[
\|\nabla\bar{v}\|_2^{N}
= 
N(2^*-1)\mu^{(N-2)/N} \J(0,v)
=
\inf \{\|\nabla w\|_2^{N}|\; w\in H ^1_0(\Omega),\|w\|_{2^*}=1\},
\]
which is a contradiction (see \cite[Proposition 1.43]{Willem}).
\end{proof}
Now we are ready to prove the following
\begin{Th}\label{MainProp}
If $\lambda\in (0,\lambda_1(\Omega))$, $\mu\in \mathbb{I}_N$, then there exists a ground state $(u,v)$ of $\J$ such that
$$\J(u,v)=\inf_{\mathcal{N}}\J=\inf_{\mathcal{N}'} \J
=\mathcal{B}.$$
\end{Th}

\begin{proof}
The functional $\J$ satisfies the geometrical assumptions of the Mountain Pass Theorem. Indeed, obviously, $\J(0,0)=0$. 
Using the Poincar\'e  and the Sobolev inequalities we have that
\[
\J(u,v)
\geq  
C(\|\nabla u\|_2^2 +\|\nabla v\|_2^2 -\|\nabla v\|_2^{2^*} -\|\nabla u\|_2^{2^*-1} \|\nabla v\|_2)
\geq  \alpha
\]
for some $\alpha>0$ and $\rho=\sqrt{\|\nabla u\|_2^2 +\|\nabla v\|_2^2}$ sufficiently small.
Moreover if $(u,v)\in H^1_0(\Omega)\times H^1_0(\Omega))$ satisfies
\begin{equation*}
\mu \| v \|_{2^*}^{2^*} + 2^* \int_{\Omega} |u|^{2^*-1}v\; dx >0,
\end{equation*}
then
\[
\mathcal{J} (tu,tv)= \frac{t^2}{2} \left( \| \nabla u \|_2^2 - \lambda  \| u \|_2^2 + \frac{1}{2^*-1}  \| \nabla v \|_2^2  \right) - \frac{t^{2^*}}{2^*-1} \left( \frac{\mu}{2^*} \| v \|_{2^*}^{2^*} + \int_{\Omega} |u|^{2^*-1}v\; dx \right)\to -\infty
\]
as $t \to +\infty$.
Then there exists a $(PS)_{\mathcal{B}}$-sequence $\{(u_n,v_n)\}\in H^1_0(\Omega)\times H^1_0(\Omega)$ for $\J$ at level $\mathcal{B}$, i.e. a sequence such that $\J(u_n,v_n)\to\mathcal{B}$ and $\J'(u_n,v_n)\to 0$.
Since for some constant $C>0$
\[
C (\|\nabla u_n\|_2^2 + \|\nabla v_n\|_2^2)
\leq 
\J(u_n,v_n)-\frac{1}{2^*}\langle\J'(u_n,v_n),(u_n,v_n)\rangle
\leq 
\mathcal{B}+1+\sqrt{\|\nabla u_n\|_2^2 + \|\nabla v_n\|_2^2},
\]
we have that the sequence $\{(u_n,v_n)\}$ is bounded.
Therefore, up to a subsequence, we may assume that there exists $(u,v)\in H^1_0(\Omega)\times H^1_0(\Omega)$ such that
\[
\begin{array}{lllll}
u_n\rightharpoonup u
& \hbox{ in } H^1_0(\Omega),
& 
& v_n\rightharpoonup v
& \hbox{ in } H^1_0(\Omega),\\
|u_n|^{2^*-1}\rightharpoonup |u|^{2^*-1}
& \hbox{ in } L^{2^*/(2^*-1)}(\Omega),
& 
& |v_n|^{2^*-2}v_n \rightharpoonup |v|^{2^*-2}v
& \hbox{ in } L^{2^*/(2^*-1)}(\Omega),\\
|u_n|^{2^*-3}u_n v_n \rightharpoonup |u|^{2^*-3}u v
& \hbox{ in } L^{2^*/(2^*-1)}(\Omega),
& 
& u_n \to u 
& \hbox{ in } L^2(\Omega),\\
u_n\to u
& \hbox{ a.e. on }\Omega,
& 
& v_n\to v
& \hbox{ a.e. on }\Omega.
\end{array}
\]
Hence, for every $(\xi,\eta)\in H^1_0(\Omega)\times H^1_0(\Omega)$, we have
\begin{multline*}
|\langle\J'(u_n,v_n),(\xi,\eta)\rangle
- \langle\J'(u,v),(\xi,\eta)\rangle |\\
= 
\left| \int_\Omega (\nabla u_n - \nabla u) \nabla \xi
- \lambda \int_\Omega (u_n - u) \xi
+\frac{1}{2^*-1} \int_\Omega (\nabla v_n - \nabla v) \nabla \eta
-\frac{1}{2^*-1} \int_\Omega (|u_n|^{2^*-1} - |u|^{2^*-1}) \eta
\right.\\
\left.
-\frac{\mu}{2^*-1} \int_\Omega (|v_n|^{2^*-2}v_n - v^{2^*-2}v) \eta
- \int_\Omega (|u_n|^{2^*-3}u_n v_n - |u|^{2^*-3}u v) \xi
\right|
 \to 0.
\end{multline*}
Thus $\J'(u,v)=0$. \\
We claim that $(u,v)\neq(0,0)$. Indeed, suppose by contradiction that $(u,v)=(0,0)$ and so
\begin{equation}
\label{tired}
u_n \to 0 \hbox{ in } L^2(\Omega).
\end{equation}
Since $\J$ is continuous and $\J(u_n,v_n)\to\mathcal{B}>0$, then $(u_n,v_n)$ cannot converge to $(0,0)$ in $H^1_0(\Omega)\times H^1_0(\Omega)$. So, up to a subsequence,
%
%
%
%
we may assume that $(u_n,v_n)\neq (0,0)$ and $\|( u_n, v_n) \|\geq C >0$ and, moreover, that $(u_{n},v_{n})\in\mathcal{A}$ for all $n\in\mathbb{N}$. Indeed, if there exists a subsequence $\{(u_{n_k},v_{n_k})\}$ of $\{(u_{n},v_{n})\}$ in $(H^1_0(\Omega)\times H^1_0(\Omega))\cap\mathcal{A}^c$, then
\[
\langle\J'(u_{n_k},v_{n_k}),(u_{n_k},v_{n_k})\rangle
\geq
\|( u_{n_k}, v_{n_k}) \|^2
\]
and, since
$$\langle\J'(u_{n_k},v_{n_k}),(u_{n_k},v_{n_k})\rangle\to 0
\hbox{ as } k\to+\infty,$$
we get a contradiction.\\
Hence, if we take  
\[
t_n=
\left[\left((2^*-1)\| \nabla u_{n} \|_2^2 +\| \nabla v_{n} \|_2^2\right)
\left(\mu \| v_{n} \|_{2^*}^{2^*} + 2^* \int_{\Omega} |u_{n}|^{2^*-1}v_{n}\right)^{-1}\right]^{\frac{1}{2^*-2}}
\]
and we denote in the same way the funcions in $H^1_0(\Omega)$ and their extensions in $\mathbb{R}^N$ putting the function equal to zero in $\mathbb{R}^N\setminus\Omega$, we have that $(t_n u_n,t_n v_n)\in\mathcal{N}_0'$ and so 
\begin{equation}
\label{proj}
\langle\J_0'(t_nu_n,t_nv_n),(t_nu_n,t_nv_n)\rangle=0.
\end{equation}
Moreover, using \eqref{tired},
\begin{equation}
\label{tired2}
\langle\J_0'(u_n,v_n),(u_n,v_n)\rangle=\langle\J'(u_n,v_n),(u_n,v_n)\rangle+o(1)=o(1).
\end{equation}
Thus, combining \eqref{proj} and \eqref{tired2} we get that  $t_n\to 1$.
Hence, taking into account Lemma \ref{LemmaStep1}, Theorems \ref{PropAandA}, \ref{cor0}, and \ref{corupp} for $N=5$, or corresponding results from Subsection \ref{subs22} for $N\geq 6$, we have
$$\mathcal{B} < A = A' \leq\lim_{n}\J(t_nu_n,t_nv_n)=\mathcal{B}$$
getting a contradiction.\\
Hence $(u,v)\neq (0,0)$ and  
$(u,v)\in\mathcal{N}
\subset \mathcal{N}'$. Similarly as above, we find $t_n\to 1$ such that 
$(t_n u_n,t_n v_n)\in\mathcal{N}'$. In view of  Proposition \ref{PropEqInfB} we get
\begin{equation*} 
\inf_{\mathcal{N}}\J\leq \J(u,v)
\leq \lim_{n\to\infty} \J(t_nu_n,t_nv_n)=\mathcal{B}= \inf_{\mathcal{N}'}\J
\leq\inf_{\mathcal{N}}\J
\end{equation*}
and we conclude.
\end{proof}

To prove that our solutions are positive, let us write $u=u_+ + u_-$, where $u_+$ and $u_-$ are respectively the positive and the negative part of $u$ and let us consider the following functional
\begin{equation}\label{DefOfJplus}
\J_+(u,v)=\frac{1}{2}\int_{\Omega}|\nabla u|^2-\frac{\lambda}{2}\int_{\Omega}|u|^2
+\frac{1}{2(2^*-1)}\int_{\Omega}|\nabla v|^2
-\frac{\mu}{2^*(2^*-1)}\int_{\Omega} v_{+}^{2^*}
-\frac{1}{2^*-1}\int_{\Omega}  u_+^{2^*-1}v
\end{equation}
which is of $C^1$ class on $H^1_0(\Omega)\times H^1_0(\Omega)$, with
\[
\langle \J'_+(u,v),(\xi,\eta) \rangle
=
\int_{\Omega}\nabla u \nabla \xi
-\lambda\int_{\Omega}  u\xi
-\int_{\Omega}u_+^{2^*-2}v \xi
+
\frac{1}{2^*-1}\int_{\Omega}\nabla v \nabla \eta -\frac{\mu}{2^*-1}\int_{\Omega}v_{+}^{2^*-1} \eta
-\frac{1}{2^*-1}\int_{\Omega}u_+^{2^*-1} \eta.
\]
We have
\begin{Lem}\label{LemPSuv_plus}
Suppose that $\{(u_n,v_n)\}$ is a $(PS)_{c}$-sequence for $\J_{+}$, with $c>0$. Then $\{(u_n,v_n)\}$ is bounded and $\{((u_n)_{+},(v_n)_{+})\}$ is also a $(PS)_{c}$-sequence for $\J_{+}$.
\end{Lem}
\begin{proof}
Let $\{(u_n,v_n)\}$ be a $(PS)_{c}$-sequence for $\J_{+}$. There exists $C>0$ such that
\[
C (\|\nabla u_n\|_2^2 + \|\nabla v_n\|_2^2)
\leq
\J_{+}(u_n,v_n)-\frac{1}{2^*}\langle\J'_{+}(u_n,v_n),(u_n,v_n)\rangle
\leq
c+1+\sqrt{\|\nabla u_n\|_2^2 + \|\nabla v_n\|_2^2}
\]
and so $\{(u_n,v_n)\}$ is bounded. Moreover
\begin{align*}
o(1)
& =
\langle \J'_{+}(u_n,v_n), ((u_n)_{-},(v_n)_{-})\rangle \\
& =
\| \nabla (u_n)_{-} \|_2^2 - \lambda \| (u_n)_{-} \|_2^2
+\frac{1}{2^*-1} \| \nabla (v_n)_{-} \|_2^2
-\frac{1}{2^*-1} \int_{\Omega}(u_n)_+^{2^*-1} (v_n)_{-}
\\
& \geq
C(\|\nabla (u_n)_{-}\|_2^2 + \|\nabla (v_n)_{-}\|_2^2),
\end{align*}
and then $((u_n)_{-},(v_n)_{-})\to (0,0)$ in $H^1_0(\Omega)\times H^1_0(\Omega)$ and
\[
\int_{\Omega}(u_n)_+^{2^*-1} (v_n)_{-}\to 0.
\]
Thus 
\begin{eqnarray*}
\J_{+}(u_n,v_n)-\J_{+}((u_n)_{+},(v_n)_{+})
&=&
\frac{1}{2}
(\|\nabla (u_n)_{-}\|_2^2-\lambda \| (u_n)_{-}\|_2^2)+
\frac{1}{2(2^*-1)}\|\nabla (v_n)_{-}\|_2^2\\
&&-\frac{1}{2^*-1}\int_{\Omega}(u_n)_+^{2^*-1} (v_n)_{-}\\
&\to& 0.
\end{eqnarray*}
Finally, since $\{(u_n)_{+}\}$ is bounded and $((u_n)_{-},(v_n)_{-})\to (0,0)$ in $H^1_0(\Omega)$, then for every $(\xi,\eta)\in H^1_0(\Omega)\times H^1_0(\Omega)$ we have

\begin{align*}
|\langle \J_{+}'(u_n,v_n)-&\J_{+}'((u_n)_{+},(v_n)_{+}),(\xi,\eta)\rangle|\\
& =
\left|\int_{\Omega}\nabla (u_n)_{-} \nabla \xi
-\lambda\int_{\Omega} (u_n)_{-}\xi
+\frac{1}{2^*-1}\int_{\Omega}\nabla (v_n)_{-} \nabla \eta-\int_{\Omega}(u_n)_+^{2^*-2} (v_n)_{-}\xi\right|\\
& \leq
C(\|\nabla (u_n)_{-}\|_2\|\nabla \xi \|_2+\|\nabla (v_n)_{-}\|_2\|\nabla \eta \|_2 +\|\nabla (u_n)_{+}\|_2^{2^*-2}\|\nabla (v_n)_{-}\|_2\|\nabla \xi\|_2)\\
& \leq
C(\|\nabla (u_n)_{-}\|_2+(1+\|\nabla (u_n)_{+}\|_2^{2^*-2})\|\nabla (v_n)_{-}\|_2)(\|\nabla \xi\|_2^2+\|\nabla \eta\|_2^2)^{\frac{1}{2}}
\end{align*}
and then
$$\|\J_{+}'(u_n,v_n)-\J_{+}'((u_n)_{+},(v_n)_{+})\|\to 0.$$
\end{proof}

\begin{proof}[Proof of Theorem \ref{MainTh}]
As in the proof of Theorem \ref{MainProp} we can show that the functional $\J_{+}$ satisfies the geometrical assumptions of the Mountain Pass Theorem. 
Then there exists a $(PS)_{\mathcal{B}}$-sequence $\{(u_n,v_n)\}\in H^1_0(\Omega)\times H^1_0(\Omega)$ for $\J_{+}$ at level $\mathcal{B}$
In view of Lemma \ref{LemPSuv_plus} we may assume that $u_n=(u_n)_{+}\hbox{ and }v_n=(v_n)_{+}$ and $\{(u_n,v_n)\}$ is bounded.
Note that $\J(u_n,v_n)=\J_{+}(u_n,v_n)$ and thus we can conclude following the arguments given in proof of Theorem \ref{MainProp}, getting a ground state $(u,v)$ of $\J$ such that $u,v\geq 0$. 
Finally the Strong Maximum Principle (see \cite[Theorem 8.19]{GT}) implies that $u,v>0$.
\end{proof}

\section{Nonexistence result}\label{SectionNonexisctence}

\begin{proof}[Proof of (\ref{131}) of Theorem \ref{ThNonexistence}]
Let $(u,v)\in H^1_0(\Omega)\times H^1_0(\Omega)$ be a nontrivial solution to \eqref{eq} for $\lambda\leq 0$ and $\mu>\tilde{\mu}_N$.
In system \eqref{eq} multiply the fist equation by $v$, the second equation by $u$ and this leads to 
\begin{equation}\label{eqNon}
\int_{\Omega} u(u^{2^*-1} -u^{2^*-3}v^2 +\mu v^{2^*-1}-\lambda v)=0
\end{equation}
Considering the function $f_N$ in \eqref{fN}, we get that $u^{2^*-1} -u^{2^*-3}v^2 +\mu v^{2^*-1}>0$ in $\Omega$
for $\mu>\tilde{\mu}_N$ and this is in a contradiction with \eqref{eqNon}.
\end{proof}

\begin{proof}[Proof of (\ref{132}) of Theorem \ref{ThNonexistence}]
Suppose that $\lambda \geq \lambda_1 (\Omega)$ and $\mu\in\R$. We proceed  similarly as in \cite[Remark 1.1]{BN} arguing only on the first equation of \eqref{eq}. Let $(u,v)\in H^1_0(\Omega)\times H^1_0(\Omega)$ be a nontrivial solution to \eqref{eq} and $\varphi_1$ the eigenfunction of $-\Delta$ with Dirichlet boundary conditions corresponding to $\lambda_1 (\Omega)$. Multiplying the first equation of \eqref{eq} by $\varphi_1$ we have
\[
-\int_{\Omega} \Delta u \varphi_1 = \lambda \int_{\Omega} u \varphi_1 + \int_{\Omega} u^{2^*-2} v \varphi_1.
\]
On the other hand
\[
-\int_{\Omega} \Delta u \varphi_1 = -\int_{\Omega}  u \Delta \varphi_1 = \lambda_1 (\Omega) \int_{\Omega} u \varphi_1
\]
and so if $\lambda \geq \lambda_1 (\Omega)$ we reach a contradiction.
\end{proof}

\begin{proof}[Proof of (\ref{133}) of Theorem \ref{ThNonexistence}]
Here we adopt Poho\u{z}aev type arguments (see e.g. \cite{K} or \cite[Appendix B]{Willem} ).
Let $\Omega \subset \RN$ be a star shaped domain and $(u,v)\in H^2(\bar{\Omega})\times H^2(\bar{\Omega})$ be nontrivial solution of \eqref{eq}. If we multiply the first equation of \eqref{eq} by $x\cdot\n u$ and the second one by $x\cdot\n v$ we have that
\begin{align*}
0=&(\Delta u + \lambda u + u^{2^*-2} v)(x\cdot\n u)\\
 =&\operatorname{div}\left[(\n u)(x\cdot\n u)\right] - |\n u|^2
   - x\cdot \n \left(\frac{|\n u|^2}{2}\right) + \frac{\lambda}{2} \left[\operatorname{div}( x u^2) - N u^2\right]\\
  &\quad + \frac{1}{2^*-1} \left[ \operatorname{div} (x u^{2^*-1} v ) - N u^{2^*-1} v - u^{2^*-1} (x\cdot\n v)\right]\\
 =&\operatorname{div}\left[(\n u)(x\cdot\n u) - x \frac{|\n u|^2}{2} + \frac{\lambda}{2} x u^2
   +\frac{1}{2^*-1} x u^{2^*-1} v \right] + \frac{N-2}{2} |\n u|^2 - \frac{N\lambda}{2}  u^2\\
  &\quad - \frac{N}{2^*-1} u^{2^*-1} v - \frac{1}{2^*-1} u^{2^*-1} (x\cdot\n v)
\end{align*}
and
\begin{align*}
0=&(\Delta v + \mu v^{2^*-2}v + u^{2^*-1})(x\cdot\n v)\\
 =&\operatorname{div}\left[(\n v)(x\cdot\n v)\right] - |\n v|^2
   - x\cdot \n \left(\frac{|\n v|^2}{2}\right) + \frac{\mu}{2^*}\left[\operatorname{div} (x v^{2^*}) - N v^{2^*} \right] + u^{2^*-1} (x\cdot\n v)\\
 =&\operatorname{div}\left[(\n v)(x\cdot\n v) - x \frac{|\n v|^2}{2} + \frac{\mu}{2^*} x v^{2^*} \right] + \frac{N-2}{2} |\n v|^2 - \frac{N\mu}{2^*} v^{2^*} + u^{2^*-1} (x\cdot\n v).
\end{align*}
Integrating on $\Omega$ and using the boundary conditions on $u$ and $v$ we obtain
\begin{equation}\label{eq:Poho1}
0=\frac{1}{2} \int_{\partial\Omega} \left| \frac{\partial u}{\partial {\bf n}}\right|^2 x\cdot {\bf n}+\frac{N-2}{2} \|\n u \|_2^2 - \frac{N\lambda}{2}  \|u\|_2^2 - \frac{N}{2^*-1} \iO u^{2^*-1} v -\frac{1}{2^*-1} \iO u^{2^*-1} (x\cdot\n v)
\end{equation}
and
\begin{equation}\label{eq:Poho2}
0=\frac{1}{2} \int_{\partial\Omega} \left| \frac{\partial v}{\partial {\bf n}}\right|^2 x\cdot {\bf n}+\frac{N-2}{2} \|\n v\|_2^2 - \frac{N\mu}{2^*} \|v\|_{2^*}^{2^*} +  \iO u^{2^*-1} (x\cdot\n v),
\end{equation}
where ${\bf n}$ is the unit exterior normal to $\partial\Omega$.\\
Moreover, multiplying the equations of \eqref{eq} respectively by $u$ and by $v$ we get
\begin{equation}\label{eq:Ne1}
\|\n u \|_2^2= \lambda  \| u \|_2^2 +  \iO u^{2^*-1} v
\end{equation}
and
\begin{equation}\label{eq:Ne2}
\|\n v\|_2^2 = \mu \|v\|_{2^*}^{2^*} +  \iO u^{2^*-1} v.
\end{equation}
Hence, combining \eqref{eq:Poho1}, \eqref{eq:Poho2}, \eqref{eq:Ne1} and \eqref{eq:Ne2}, we have
\begin{equation}
\label{eq:finPoho}
-\lambda \| u \|_2^2 + \frac{1}{2} \int_{\partial\Omega} \left| \frac{\partial u}{\partial {\bf n}}\right|^2 x\cdot {\bf n}  + \frac{1}{2(2^*-1)}  \int_{\partial\Omega} \left| \frac{\partial v}{\partial {\bf n}}\right|^2 x\cdot {\bf n} = 0
\end{equation}
Then, if $\lambda<0$ we get a contradiction.\\
If $\lambda=0$, from \eqref{eq:finPoho} we have
\[
\frac{\partial u}{\partial {\bf n}}= \frac{\partial v}{\partial {\bf n}}=0
\hbox{ on } \partial\Omega
\]
and so, using the positivity of $u$ and $v$ and the first equation of \eqref{eq} we get a contradiction.
\end{proof}

\subsection*{Acknowledgements} We would like to thank Giusi Vaira for her valuable comments concerning the problem in case $N\geq 6$.

\end{document}